\documentclass[12pt,a4paper,reqno]{amsart} 
\usepackage{amssymb}
\usepackage{latexsym}
\usepackage{amsmath}
\usepackage{mathrsfs}
\usepackage[X2,T1]{fontenc}
\usepackage{cite}

\usepackage{calc}                   

\def\Mopq{M_{(\omega)} ^{p,q}}

 \def\cS{\mathcal{S}}
 
 \def\cD{\mathcal{D}}

 \newcommand{\stft}{short-time Fourier transform}

\newcommand{\loc}{\operatorname{loc}}
\newcommand{\scal}[2]{\langle #1,#2\rangle}
\newcommand{\rr}[1]{\mathbf R^{#1}}

\newcommand{\nm}[2]{\Vert #1\Vert _{#2}}

\newcommand{\sets}[2]{\{ \, #1\, ;\, #2\, \} }
\newcommand{\ep}{\varepsilon}
\newcommand{\fy}{\varphi}
\newcommand{\cdo}{\, \cdot \, }
\newcommand{\supp}{\operatorname{supp}}

\newcommand{\eabs}[1]{\langle #1\rangle}     

\newcommand{\vrum}{\vspace{0.1cm}}
\newcommand{\ttbigcap}{{\textstyle{\, \bigcap \, }}}

\DeclareMathOperator{\WF}{WF}
\DeclareMathOperator{\DF}{DF}
\newcommand{\MRs}{\mathscr M_{\{s\}}}

\newcommand{\Mv}{\mathscr M_{v}}
\numberwithin{equation}{section}          

\newtheorem{thm}{Theorem}
\numberwithin{thm}{section}

\newcommand{\rubrik}{}
\newtheorem{prop}[thm]{Proposition}
\newtheorem{cor}[thm]{Corollary}
\newtheorem{lemma}[thm]{Lemma}

\theoremstyle{definition}

\newtheorem{defn}[thm]{Definition}

\theoremstyle{remark}
\newtheorem{rem}[thm]{Remark}              

\author{Karoline Johansson}

\address{Department of Computer science, Mathematics and Physics,
Linn{\ae}us University, V{\"a}xj{\"o}, Sweden
}

\email{karoline.johansson@lnu.se}

\author{Stevan Pilipovi\' c}

\address{Department of Mathematics and Informatics,
University of Novi Sad, Novi Sad, Serbia}

\email{stevan.pilipovic@dmi.uns.ac.rs}

\author{Nenad Teofanov}

\address{Department of Mathematics and Informatics,
University of Novi Sad, Novi Sad, Serbia}

\email{nenad.teofanov@dmi.uns.ac.rs}

\author{Joachim Toft}

\address{Department of Computer science, Mathematics and Physics,
Linn{\ae}us University, V{\"a}xj{\"o}, Sweden
}

\email{joachim.toft@lnu.se}

\title{\textbf {Micro-local analysis in some spaces of ultradistributions}}

\keywords{Wave-front sets, weighted Fourier-Lebesgue spaces,
ultradistributions, pseudo-differential operators, micro-local
analysis}

\subjclass[2000]{35A18,35S30,42B05,35H10}

\begin{document}

\begin{abstract}
In this paper we extend some results from \cite{JPTT1} and
\cite{PTT}, concerning wave-front sets of Fourier-Lebesgue and modulation space types,
to a broader class of spaces of ultradistributions, and relate these wave-front sets with
the usual wave-front sets of ultradistributions.

\par

Furthermore, we use Gabor frames for the description of
discrete wave-front sets, and prove that these wave-front sets coincide with
corresponding continuous ones.
\end{abstract}

\maketitle

\section{Introduction} \label{sec0}

Wave-front sets with respect to Fourier Lebesgue and
modulation spaces were introduced in \cite{PTT} and studied further in \cite{PTT2, PTT3, PTT4}.
Among other properties, it was proved that wave-front sets of Fourier Lebesgue and
modulation spaces agree with each others, and that the usual
wave-front sets with respect to smoothness (cf. \cite[Sections
8.1--8.3]{Hrm-nonlin}) can be obtained as wave-front sets of sequences of
Fourier Lebesgue or modulation spaces.
Discrete versions of wave-front sets of Fourier Lebesgue and
modulation spaces, related to those wave-front sets in \cite{RT1}, were introduced and studied in
\cite{JPTT1}. In particular, it was proved that these agree with corresponding continuous ones.

\par

In this paper we extend some results from \cite{JPTT1} and
\cite{PTT} to a broader class of spaces of ultradistributions.
Instead of polynomial growth at infinity, here we study objects
which may have almost exponential growth at infinity.
We may thereby recover the usual wave-front sets of ultradistributions, see Section \ref{sec2}.

\par

Furthermore, we introduce discrete versions of wave-front sets of ultradistributions and prove
that these wave-front sets coincide with corresponding continuous ones (cf. \cite{JPTT1}).
The results in form of series, established when introducing
these discrete versions, might be
useful for numerical analysis of micro-local properties of functions and
ultradistributions. For example, we use Gabor frames for the description of
discrete wave-front sets. (See \cite{FS1, FS2} for numerical treatment of Gabor frame theory.) With
that respect we emphasize that in the process of analysis and
synthesis of a signal, Gabor frame coefficients also give
information on micro-local properties of the signal.

\par

Since compactly supported smooth functions are used in the process
of microlocalization we are limited to the use of weights with
almost exponential growth at infinity described by the
Beurling-Domar condition. 
We refer to  subsection \ref{notions} for the notions and to  \cite{Gro2}
for a discussion on the role of weights in time-frequency
analysis.

\par

Our investigation can therefore be considered as the starting point in
the study of  analytic wave-front sets and pseudodifferential operators with ultrapolynomial
symbols, known also as symbol-global type operators. This will be done in a separate paper.

\par

The paper is organized as follows. In Section \ref{sec1} we introduce wave-front sets of
Fourier Lebesgue types with respect to ultradistributions. These are compared to other
types of wave-front sets in Section \ref{sec2}. In Section \ref{sec3} we discuss wave-front
sets of modulation space types and show equivalence with those of Fourier Lebesgue
types. Sections \ref{sec4}--\ref{sec6} are devoted to discrete versions of wave-front sets.
For the reader's convenience the results in Sections \ref{sec1}--\ref{sec6} are formulated in
terms of ultradistributions of Roumieu  type and necessary explanations
concerning the Beurling type ultradistributions and differences between the Roumieu
and Beurling cases are given in the last section.

\par

\subsection{Basic notions and notation} \label{notions}

\par

In this subsection we collect some notation and notions which will be used in the sequel.

\par

We put $\mathbf N =\{0,1,2,\dots \}$, $\mathbf Z_+=\{1,2,3,\dots\}$, $\eabs x =(1+|x|^2)^{1/2}$,
for $x\in \rr d$, and $A\lesssim B$ to
indicate $A\leq c B$ for a suitable constant $c>0$. The
symbol $B_1 \hookrightarrow B_2$ denotes the continuous and dense embedding of
the topological vector space $B_1$ into $B_2$. The scalar product in $L^2 $ is denoted  by
$(\cdot , \cdot)_{L^2}.$
Translation and modulation operators are given by
\begin{equation}
  \label{eqi1}
 T_xf(t)=f(t-x)\quad{\rm and}\quad M_{\xi}f(t)= e^{i \eabs{\xi, t}}f(t) \, .
\end{equation}

\par

\subsubsection{Weights}

\par

In general, a weight function is a non-negative function.

\begin{defn}
Let $\omega, v$ be non-negative functions. Then 
\begin{enumerate}
\item $v$ is called \emph{submultiplicative} if
\[
v(x+y)\leq v(x)v(y), \qquad \forall \ x,y\in\rr{d};
\]
\item $\omega$ is called $v$-\emph{moderate} if 
\[
\omega(x+y) \lesssim v(x)\omega (y), \qquad \forall \ x,y\in\rr{d}.
\]
\end{enumerate}
For a given submultiplicative weight $v$  the set of all
$v$-moderate weights will be denoted by $\Mv$.
\end{defn}

\par

If $\omega \in \Mv $, then $1/v \lesssim \omega \lesssim v$,  $\omega \neq 0$ everywhere and
$1/\omega \in \Mv$.

\par

In the sequel $v$ will always stand for a submultiplicative function.
Submultiplicativity implies that $v$ is dominated by an exponential function, i.e.
\begin{equation} 
 \exists\, C, k>0 \quad \mbox{such\, that}\quad  v \leq C e^{k |\cdo|}.
\end{equation}

%

\par

A submultiplicative weight $v$ satisfies the GRS condition (the
Gelfand-Raikov-Shilov condition) if $ \displaystyle \lim_{n\to
\infty} v(n x)^{1/n} =1$, for every $x\in \rr d$.


\par

Let $s>1$.
By $\MRs (\rr d)$ we denote the set of  all weights which are moderate
with respect to a weight $v$ which satisfies $ v \leq Ce^{k|\cdo|^{1/s}}$
for some positive constants $C$ and $k$.
The weight $v$ satisfy the Beurling-Domar non-quasi-analyticity condition which takes the form
\[
\sum\limits_{n=0 } ^{\infty}\frac{\log v(nx)}{n^2} < \infty, \quad  x \in \rr d,
\]
and which is  stronger than the Gelfand-Raikov-Shilov condition, cf. \cite{Gro2}.

\par

\subsubsection{Test function spaces and their duals}

\par

Next we introduce spaces of test functions and their duals in the
context of spaces of ultradistributions. We start with
Gelfand-Shilov type spaces.

\par

\begin{defn}
Let  $ s>1$ and $A>0$. We denote by  $\cS^{s}_{A}(\rr d)$ the space
of all functions $\fy \in C^{\infty}(\rr d)$ such that the norm 
\[
\nm{\fy}{s,A} =\sup_{\alpha,\beta\in \mathbf N_0^d} \sup_{x\in \rr
d} \frac{A^{|\alpha + \beta| }}{\alpha!^s\beta!^s}
\eabs{x}^{|\alpha|}|\fy^{(\beta)}(x)|
\]
is finite. Then the projective limit is denoted by $\cS^{(s)}(\rr d)$, i.{\,}e., 
\[
\cS^{(s)}(\rr d) = \operatorname{proj}\lim_{A\to \infty}
\cS^{s}_{A}(\rr d),
\]
and the inductive limit is denoted by $\cS^{\{s\}}(\rr d)$, i.{\,}e.,
\[
\cS^{\{s\}} (\rr d) =\operatorname{ind}\lim_{A\to 0} \cS^{s}_{A}(\rr
d).
\]
The space $\cS^{\{s\}}$ is called the Gelfand-Shilov space of order $s$.
\end{defn}

\par

The strong dual spaces of $\cS^{(s)} (\rr d)$ and $\cS^{\{s\}}
(\rr d)$ are spaces of tempered ultradistributions of Beurling and
Roumieu type denoted by $(\cS^{(s)})' (\rr d)$ and $(\cS^{\{s\}})'
(\rr d),$ respectively. If $s>t,$ then
\begin{eqnarray*}
\cS^{(t)}(\rr d) & \hookrightarrow &  \cS^{\{t\}} (\rr d)
\hookrightarrow \cS^{(s)}(\rr d) \hookrightarrow  \cS^{\{s\}} (\rr  d) \\
& \hookrightarrow &
(\cS^{\{s\}})' (\rr  d) \hookrightarrow (\cS^{(s)})'(\rr d)
\hookrightarrow  (\cS^{\{t\}})' (\rr d) \hookrightarrow (\cS^{(t)})' (\rr d).
\end{eqnarray*}

\par

In order to perform (micro)local analysis we use the following
test function spaces on open sets, cf. \cite{K1}.

\par

\begin{defn} Let $ X $ be an open set in $ \rr d$. For a given compact set $K
\subset X$, $ s> 1$ and  $ A>0$  we denote by $ \mathcal{E}^{s} _{A,K} $
the space of all  $\fy \in C^{\infty}(X)$ such that the norm
\begin{equation} \label{ultra-norm}
\| \fy \|_{s,A,K} = \sup_{\beta\in \mathbf N_0^n} \sup_{x\in K}
\frac{A^{| \beta |}}{\beta!^s} |\fy^{(\beta)}(x)|
\end{equation}
is finite.

We denote by $ \mathcal{E}^{s} _{A} (K) $ the space of functions
$\fy \in C^{\infty}(X)$ such that \eqref{ultra-norm} holds and $ \supp \fy \subseteq K.$

Let $(K_n)_n$ be a sequence of compact sets such that
$K_n\subset\subset K_{n+1}$ and $\bigcup K_n =X$. Then
\begin{align*}
\mathcal{E}^{(s)}(X) &= \operatorname{proj}\lim_{n\to \infty}
(\operatorname{proj}\lim_{A\to \infty} \mathcal{E}^{s} _{A,K_n}),\\
\mathcal{E}^{\{s\}}(X)&=\operatorname{proj}\lim_{n\to \infty}
(\operatorname{ind}\lim_{A\to 0} \mathcal{E}^{s} _{A,K_n}),\\
\mathcal{D}^{(s)}(X) &= \operatorname{ind}\lim_{n\to \infty}
(\operatorname{proj}\lim_{A\to \infty} \mathcal{E}^s _A (K_n)),\\
\intertext{and} 
\mathcal{D}^{\{s\}}(X)&=\operatorname{ind}\lim_{n\to \infty}
(\operatorname{ind}\lim_{A\to 0} \mathcal{E}^s _A (K_n)).
\end{align*}
\end{defn}

The spaces of linear functionals over  $\mathcal{D}^{(s)}(X)$ and
$\mathcal{D}^{\{s\}}(X)$, denoted by $ (\mathcal{D}^{(s)})'(X)$
and $ (\mathcal{D}^{\{s\}})'(X)$ respectively, are called the
spaces of {\em ultradistributions} of Beurling and Roumieu  type
respectively, while the spaces of linear functionals over  $
\mathcal{E}^{(s)}(X)$ and    $ \mathcal{E}^{\{s\}}(X)$, denoted by
$ (\mathcal{E}^{(s)})'(X)$ and $ (\mathcal{E}^{\{s\}})'(X)$,
respectively are called the spaces of {\em ultradistributions
of compact support} of  Beurling  and Roumieu type respectively,
\cite{K1}. We have
\begin{align*}(\mathcal{E}^{ \{s\} })'(X) &\subset
(\mathcal{E}^{(s)})'(X),\quad
 (\mathcal{E}^{(s)})'(X) \subset
(\mathcal{E}^{(s)})'(\rr d) \quad \text{and}\\ (\mathcal{E}^{\{s\}})'(X)
&\subset (\mathcal{E}^{\{s\}})'(\rr d).
\end{align*}

\par

Clearly, 

\begin{align*} (\mathcal{E}^{\{s\}})' (\rr d)  &\subset
(\cS^{\{s\}})' (\rr d) 
\subset
(\mathcal{D}^{\{s\}})' (\rr d)\\
 \intertext{and} 
 (\mathcal{E}^{(s)})'(\rr d) &\subset
(\cS^{(s)})' (\rr d)
\subset
(\mathcal{D}^{(s)})' (\rr d).
\end{align*}
 Any  ultra-distribution with compact support
can be viewed as an element of $ (\cS^{(1)})' (\rr d)$.

\par

Obviously, $\mathcal{D}^{(s)}(X)$  ($\mathcal{D}^{\{s\}}(X)$  resp.) are subspaces of
$ \mathcal{E}^{(s)} (X) $ (of $ \mathcal{E}^{\{s\}} (X) $ resp.)
whose elements are compactly supported. We also remark that a usual notation for
the space $\mathcal{D}^{\{s\}}(X) $ is  $ G^{s}(X) $ (cf. \cite{R}).

\par

%
%
%
%
%
%

\subsubsection{Fourier-Lebesgue spaces}

\par

The Fourier transform $\mathscr F$ is the linear and continuous
mapping on $\mathscr S'(\rr d)$ which takes the form
\[
(\mathscr Ff)(\xi )= \widehat f(\xi ) \equiv (2\pi )^{-d/2}\int _{\rr
{d}} f(x)e^{-i\scal  x\xi }\, dx
\]
when $f\in L^1(\rr d)$. It is a homeomorphism on $(\cS^{\{s\}})' (\rr d)$
(on $(\cS^{(s)})' (\rr d)$ resp.) which restricts to a homeomorphism on $\cS^{\{s\}} (\rr d)$
(on $\cS^{(s)} (\rr d)$ resp.)  and
to a unitary operator on $L^2(\rr d)$.

\medspace

Let $q\in [1,\infty ]$, $s>1$ and $\omega \in \MRs (\rr d)$.
The (weighted) Fourier Lebesgue space $\mathscr
FL^q_{(\omega )}(\rr d)$ is the inverse Fourier image of
$L^q _{(\omega )} (\rr d)$, i.{\,}e. $\mathscr FL^q_{(\omega )}(\rr d)$
consists of all $f\in (\mathcal S ^{\{ s \}})'(\rr d)$ such that
\begin{equation}\label{FLnorm}
\nm f{\mathscr FL^{q}_{(\omega )}} \equiv \nm {\widehat f\cdot \omega
}{L^q} .
\end{equation}
is finite. If $\omega =1$, then the notation $\mathscr FL^q$
is used instead of $\mathscr FL^q_{(\omega )}$. We note that if
$\omega (\xi )=\eabs \xi ^s$, then $\mathscr
FL^{q}_{(\omega )}$ is the Fourier image of the Bessel potential space
$H^p_s$. 

\par

\begin{rem} \label{x-independence}
In many situations it is convenient to permit an $x$
dependency for the weight $\omega$ in the definition of Fourier
Lebesgue spaces. More precisely, for each $\omega \in \MRs (\rr {2d})$
we let $\mathscr FL^q_{(\omega )}$ be the set of all ultradistributions $f$ such that
\[
\nm f{\mathscr FL^q_{(\omega )}}
\equiv \nm {\widehat f\, \omega (x,\cdo )}{L^q}
\]
is finite. Since $\omega$ is $v$-moderate it follows that
different choices of $x$ give rise to equivalent norms.
Therefore the condition $\nm f{\mathscr FL^{q}_{(\omega )}}<\infty$ is
 independent of $x$, and  it follows that $\mathscr FL^q_{(\omega )}(\rr
d)$ is  independent of $x$ although $\nm \cdo {\mathscr
FL^{q}_{(\omega )}}$ might depend on $x$.
\end{rem}

\par

\par

\section{Wave-front sets of Fourier-Lebesgue type in spaces of Roumieu
type ultradistributions}\label{sec1}

\par

In this section we introduce wave-front sets of Fourier-Lebesgue type in spaces of ultradistributions
 of Roumieu type.

\par

Let $ s>1$, $q\in [1,\infty ]$, and  $\Gamma \subseteq \rr d\setminus 0$ be an open cone.
If $f\in (\cS^{\{s\}})'(\rr d)$ and $\omega \in \MRs (\rr {2d})$
we define
\begin{equation}
|f|_{\mathscr FL^{q,\Gamma}_{(\omega )}} = |f|_{\mathscr
FL^{q,\Gamma}_{(\omega ), x}}
\equiv
\Big ( \int _{\Gamma} |\widehat f(\xi )\omega (x,\xi )|^{q}\, d\xi
\Big )^{1/q}\label{skoff1}
\end{equation}
(with obvious interpretation when $q=\infty$). We
note that $|\cdo |_{\mathscr FL^{q,\Gamma}_{(\omega ), x}}$ defines
a semi-norm on $(\cS^{\{s\}})'(\rr d)$ which might attain the value $+\infty$.
Since $\omega $ is $v$-moderate  it follows that different $x \in \rr d$ gives rise to
equivalent semi-norms $ |f|_{\mathscr FL^{q,\Gamma}_{(\omega ),
x}}$. Furthermore, if $\Gamma =\rr d\setminus 0$, $f\in
\mathscr FL^{q}_{(\omega )}(\rr d)$ and $q<\infty$, then
$|f|_{\mathscr FL^{q,\Gamma}_{(\omega ), x}}$ agrees with the Fourier
Lebesgue norm $\nm f{\mathscr FL^{q}_{(\omega ), x}}$ of $f$.

\par

For the sake of notational convenience we set
\begin{equation} \label{notconv1}
\mathcal B=\mathscr FL^q_{(\omega )}=\mathscr FL^q_{(\omega )} (\rr
d), \quad
\mbox{and}
\quad
|\cdo |_{\mathcal B(\Gamma )}=|\cdo |_{\mathscr
FL^{q,\Gamma}_{(\omega ), x}}.
\end{equation}
We let $\Theta _{\mathcal B}(f)=\Theta _{\mathscr FL^{q} _{(\omega
)}} (f)$ be the set of all $ \xi \in \rr d\setminus 0 $ such that
$|f|_{\mathcal B(\Gamma )} < \infty$, for some open conical
neighbourhood $\Gamma = \Gamma_{\xi}$ of $\xi$.  We also let
$\Sigma_{\mathcal B} (f)$ be the complement of $ \Theta_{\mathcal
B} (f)$ in $\rr d\setminus 0 $. Then
$\Theta_{ \mathcal B} (f)$ and $\Sigma_{\mathcal
B} (f)$ are open respectively
closed subsets in $\rr d\setminus 0$, which are independent of
the choice of $ x \in \rr d$ in \eqref{skoff1}.

\par

\begin{defn}\label{wave-frontdef1}
Let $ s>1$, $q\in [1,\infty ]$, $\mathcal B$ be as in \eqref{notconv1}, and let
$X$ be an open subset of $\rr d$. If
$\omega \in \MRs (\rr {2d})$,
then the wave-front set of ultradistribution
$f\in (\cD^{(s)})' (\rr d)$,
$
\WF _{\mathcal B}(f)  \equiv  \WF _{\mathscr FL^q_{(\omega )}}(f)
$
with respect to $\mathcal B$ consists of all pairs $(x_0,\xi_0)$ in
$X\times (\rr d \setminus 0)$ such that
$
\xi _0 \in  \Sigma _{\mathcal B} (\fy f)
$
holds for each $\fy \in \mathcal{D}^{(s)} (X)$
 such that $\fy (x_0)\neq
0$.
\end{defn}

\par

We note that $\WF  _{\mathcal B}(f)$ is a closed set in $\rr d\times
(\rr d\setminus 0)$, since it is obvious that its complement is
open. We also note that if $ x\in \rr d$ is fixed and $\omega _0(\xi
)=\omega (x,\xi )$, then $\WF _{\mathcal B} (f)=\WF _{\mathscr
FL^q_{(\omega _0)}}(f)$, since $\Sigma _{\mathcal B}$ is independent
of $x$.

\par

The following theorem shows that wave-front sets with respect to
$\mathscr FL^q_{(\omega )}$ satisfy appropriate micro-local
properties. It also shows that such wave-front sets are decreasing
with respect to the parameter $q$, and increasing with respect to the
weight $\omega$.

\par

\begin{thm}\label{wavefrontprop11}
Let $s>1$, $q,r\in [1,\infty ]$, $X$ be an open set in $\rr d$ and
$\omega ,\vartheta \in \MRs (\rr {2d})$
be such that
\begin{equation}\label{qandomega1}
r\le q,\quad \text{and}\quad \omega
(x,\xi )\lesssim \vartheta (x,\xi ).
\end{equation}
Also let  $\mathcal B$  be as in \eqref{notconv1} and put
$ {\mathcal B _0} =\mathscr FL^r_{(\vartheta )}
(\rr d)$. If $f\in (\cD^{(s)})' (\rr d)$ and  $\fy \in \mathcal{D}^{(s)}(X)$
then
\[
\WF _{ {\mathcal B} }(\fy \, f)\subseteq \WF _{\mathcal B_0}(f).
\]
\end{thm}

\par

\begin{proof}
By the definitions it is sufficient to prove
\begin{equation}\label{chi-subsetAA1}
\Sigma_{ {\mathcal B } } (\fy f) \subseteq
\Sigma_{\mathcal B_0} (f)
\end{equation}
when $ \fy \in  \mathcal{D}^{(s)}(X)$, $\vartheta=\omega$, and $f\in (\mathcal E^{(s)})'(\rr
d)$, since the statement only involves local assertions. For the
same reasons we may assume that $\omega (x,\xi ) =\omega (\xi )$ is
independent of $x$. 
Finally, we prove the assertion for $ r \in [1, \infty ).$ The case $ r = \infty $ follows by similar arguments
and is left to the reader. 

\par

Choose open cones $\Gamma _1$ and $\Gamma_2$ in $\rr d$ such that
$\overline {\Gamma _2}\subseteq \Gamma _1$.
We will use the fact that  if $f\in (\mathcal E^{(s)})'(\rr d)$ then
$|\widehat f(\xi )\omega(\xi )|\lesssim e^{k|\xi|^{1/s}}$ for some  $k>0$ and prove that
for every $N>0$, there exist $C_N>0$ such that
\begin{multline}\label{cuttoff1}
|\fy f|_{ {\mathcal B} (\Gamma _2)}\le C_N \Big (|
f|_{\mathcal B_0(\Gamma _1)} +\sup _{\xi \in \rr
d} \big ( |\widehat f(\xi )\omega(\xi )|e^{-N|\xi|^{1/s}} \big )
\Big )
\\[1ex]
\text{when}\quad \overline \Gamma _2\subseteq \Gamma_1.
\end{multline}

\par

Since $\omega \in \MRs (\rr {2d})$ by letting $F(\xi )=|\widehat f(\xi )\omega
(\xi ) |$
and $\psi (\xi )=|\widehat \fy (\xi )v (\xi )|$ we have
\begin{multline*}
|\fy f| _{ {\mathcal B}(\Gamma _2)} = \Big (\int
_{\Gamma _2}|\mathscr F(\fy f)(\xi )\omega(\xi )|^{r}\, d\xi
\Big )^{1/r}
\\[1ex]
\lesssim\Big (\int _{\Gamma _2}\Big ( \int _{\rr {d}} \psi (\xi -\eta
)F(\eta )\, d\eta \Big )^{r}\, d\xi \Big )^{1/r} \lesssim J_1+J_2,
\end{multline*}
where
\begin{align*}
&J_1  = \Big (\int _{\Gamma _2}\Big ( \int _{\Gamma _1}\psi (\xi
-\eta )F(\eta )\, d\eta \Big )^{r}\, d\xi \Big )^{1/r},
\\[1ex]
&J_2 = \Big (\int _{\Gamma _2}\Big ( \int _{\complement \Gamma _1}\psi
(\xi -\eta )F(\eta )\, d\eta \Big )^{r}\, d\xi \Big )^{1/r}.
\end{align*}
\par

Let $q_0$ be chosen such that $1/q_0+1/q=1+1/r$, and let $\chi
_{\Gamma _1}$ be the characteristic function of $\Gamma _1$. Then
Young's inequality gives
\begin{multline*}
J_1 \le  \Big (\int _{\rr {d}} \Big ( \int _{\Gamma _1}\psi (\xi
-\eta )F(\eta )\, d\eta \Big )^{r}\, d\xi \Big )^{1/r}
\\[1ex]
=\nm {\psi *(\chi _{\Gamma _1} F)}{L^{r}} \le \nm \psi {L^{q_0}}\nm
{\chi _{\Gamma _1}
F}{L^{q}} = C_\psi |f|_{\mathcal B_0(\Gamma _1)},
\end{multline*}
where $C_\psi = \nm \psi {L^{q_0}}<\infty$. 
If $\fy \in \mathcal{D}^{(s)}(X)$, then
for every $ N>0$ there exist $ C_N>0 $ such that
\begin{equation}\label{psiexpest}
\psi (\xi) = |\widehat \fy (\xi )v (\xi )| \leq  C_{{N}} e^{-(N+k)|\xi|^{1/s}} e^{k|\xi|^{1/s}} \leq
C_{{N}} e^{-N|\xi|^{1/s}}.
\end{equation}

\par

In order to estimate $J_2$, we note that  $\overline {\Gamma _2}
\subseteq \Gamma _1$ implies that
\begin{multline}\label{estimate1}
|\xi -\eta |^{1/s}>c\max
(|\xi|^{1/s},|\eta |^{1/s})
\\[1ex]
\geq \frac c2 (|\xi|^{1/s}+|\eta|^{1/s}),\qquad \xi \in
\Gamma _2,\  \eta \notin \Gamma _1
\end{multline}
holds for some constant $c>0$, since this is true when
$1=|\xi |\ge |\eta|$. A combination of \eqref{psiexpest} and \eqref{estimate1} implies
that for every $N_1>0$ we have
\[
\psi(\xi-\eta) \lesssim C e^{-2N_1(|\xi|^{1/s}+|\eta|^{1/s})}.
\]
This gives
\begin{multline*}
J_2 \lesssim \Big (\int _{\Gamma _2}\Big ( \int _{\complement
\Gamma _1}e^{-2N_1(|\xi|^{1/s}+|\eta|^{1/s})} F(\eta )\, d\eta \Big
)^{r}\, d\xi \Big )^{1/r}
\\[1ex]
\lesssim \Big (\int _{\Gamma _2}\Big ( \int _{\complement \Gamma_1}
e^{-2N_1(|\xi|^{1/s}+|\eta|^{1/s})} e^{N_1 |\eta|^{1/s}}(e^{-N_1 |\eta|^{1/s}}
F(\eta))\, d\eta \Big )^{r}\, d\xi \Big )^{1/r}
\\[1ex]
\lesssim \sup _{\eta \in \rr d}|e^{-N_1|\eta|^{1/s}}  F(\eta ))|,
\end{multline*}
which proves \eqref{cuttoff1} and the result follows.
%
\end{proof}

\par

\section{Comparisons to other types of wave-front sets} \label{sec2}

\par

Let $\omega \in \MRs (\rr {2d})$ be moderated with
respect to a weight of polynomial growth at infinity and let $f
\in \mathcal{D}' (X). $ Then $ \WF _{\mathscr FL^q_{(\omega
)}}(f)$ in Definition \ref{wave-frontdef1} is the same as the
wave-front set introduced in \cite[Definition 3.1]{PTT}.
Therefore, the information on regularity in the background of
wave-front sets of Fourier-Lebesgue type in Definition
\ref{wave-frontdef1} might be compared to the information
obtained from the classical wave-front sets, cf. Example 4.9 in
\cite{PTT}.

\par

Next we compare the wave-front sets introduced in Definition
\ref{wave-frontdef1} to the wave-front sets in spaces of
ultradistributions given in \cite{Ho1, P, R}.

\par

Let $ s>1$ and let $X$ be an open subset of $\rr d$. The ultradistribution
$f \in (\mathcal{D}^{(s)})' (\rr d)$
is $(s)$-micro-regular ($\{ s \}$-micro-regular) at $ (x_0, \xi_0)$ if there exists
$\fy \in \mathcal{D}^{(s)} (X)$ ($\fy \in \mathcal{D}^{\{ s \} } (X)$) such that $\fy (x) = 1$
in a neighborhood of $x_0$ and an open cone
$\Gamma $ which contains $\xi_0$ such that for each $k>0$ (for some $k>0$) there
is a $C> 0$ such that
\begin{equation} \label{Rodino wf set}
|\widehat{ \fy f}(\xi )| \leq C e^{-k|\xi|^{1/s}}, \quad \xi \in \Gamma.
\end{equation}
The $(s)$-wave-front set ($\{ s \}$-wave-front set) of $f$, $ \WF_{(s)} (f) $ ($ \WF_{\{ s \}} (f) $) is defined
as the complement in $X \times \rr d \setminus 0 $ of the set of all $ (x_0, \xi_0) $ where $f$ is
$(s)$-micro-regular ($\{ s \} $-micro-regular), cf. \cite[Definition 1.7.1]{R}.

\par

The $\{ s\}$-wave-front set  $ \WF_{\{ s\} } (f) $ can be found in \cite{P} and agrees
with certain wave-front set $ \WF_L (f) $ introduced in \cite[Chapter 8.4]{Ho1}.


\par

\begin{rem}
Let $s>1$, $f\in (\cD^{\{s\}})'(\rr d)$, $\fy\in \mathcal{E}^{\{s\}}(\rr d)$ and $\fy_0 \in \cD^{(s)}(X)$
be such that $\fy(x)=1$ in a neighborhood $\supp \fy _0$.
Also let $\Gamma_0, \Gamma$ be open cones such that $\overline{\Gamma_0}\subseteq
\Gamma$. If \eqref{Rodino wf set} holds for some $k,C>0$, then it follows
by straight-forward computations, using similar arguments as in the proof of
Theorem \ref{wavefrontprop11} that
\eqref{Rodino wf set} is still true for some $k,C>0$ after $\fy$ has been replaced by
$\fy_0$. Hence it follows that the following conditions are equivalent:
\begin{enumerate}
\item $(x_0,\xi_0) \not\in \WF_{\{Ês\}Ê}(f)$;
\item for some $\fy\in \cD^{ \{ s \} }(X)$, some conical neighborhood $\Gamma$ of $\xi$
such that $\fy(x_0)=1$ in a neighborhood of $x_0$ and some $C,k>0$, it follows that
\eqref{Rodino wf set} holds;

\vrum

\item for some $\fy\in \cD^{(s)}(X)$, some conical neighborhood $\Gamma$ of $\xi$ such
that $\fy(x_0)=1$ in a neighborhood of $x_0$ and some $C,k>0$, it follows that
\eqref{Rodino wf set} holds.
\end{enumerate}

\par

Consequently we may always choose $\fy$ in $\cD^{(s)}(X)$ in the definition of $\WF_{\{Ês\}Ê}(f)$.
\end{rem}

\par

In most of our considerations we are concerned with $\{ s\}$-micro-regularity. For this reason we set
$\WF _s(f) =\WF _{\{ s\} }(f)$ when $f\in (\cD ^{(s)})'$.

\par

\begin{prop} \label{s-WF set}
Let $q \in [1,\infty],$ $ s>1$, and let $ \omega_\ep (\xi) \equiv e^{k |\xi|^{1/s}} $ for
$ \xi \in \rr d$ and  $k>0$. If $f \in (\mathcal{D}^{(s)})' (\rr d)$ then
\begin{equation}
\bigcap_{k >0} \WF _{\mathscr FL^q _{(\omega_k )}} (f) =  \WF_{s} (f).
\label{intersection}
\end{equation}
\end{prop}

\par

\begin{proof}
Recall that when $ k $ is fixed, the set $ \WF _{\mathscr FL^q _{(\omega_k )}} (f) $ is defined via
$\fy \in \mathcal{D}^{(s)} (X)$, cf. Definition \ref{wave-frontdef1}.

\par

Therefore the set $ \bigcap_{k>0} \WF _{\mathscr FL^q _{(\omega_k )}} (f) $ is the complement of the set of points
$ (x_0, \xi_0) \in \rr {2d}$ for which there exists $k>0$,
$\fy \in \mathcal{D}^{(s)} (X)$ such that $\fy (x) = 1$ in a neighborhood of $x_0$ and an open cone $ \Gamma $
which contains $\xi_0$ such that
\begin{equation}
\Big ( \int _{\Gamma} |\widehat{\fy f}(\xi ) e^{k|\xi|^{1/s}}|^{q}\, d\xi \Big )^{1/q} < \infty.
\label{regularpoints}
\end{equation}

The assertion is obviously true when $q = \infty $.


\par

Therefore, let $ q\in [1,\infty) $ and assume that
$ (x_0,\xi_0)\not\in \WF _{\mathscr FL^\infty _{(\omega_k)}}  (f)
$ for some $k>0$. Then for any $ \varepsilon >0 $ such that $ k -
\varepsilon >0$ we have
$$
\int _{\Gamma} |\widehat{ \fy f}(\xi )\omega_{k-\varepsilon} (\xi )|^{q}\, d\xi
\leq
\sup_{\xi \in \Gamma}\big ( |\widehat{ \fy f}(\xi )|^q e^{kq | \xi|^{1/s}} \big )
\int _{\Gamma} e^{-\varepsilon   |\xi|^{1/s}} \, d\xi < \infty,
$$
which means that 
$$
 (x_0,\xi_0)\not\in  \bigcap_{k>0} \WF _{\mathscr FL^q _{(\omega_k)}}  (f) 
 $$
when 
$$
 (x_0,\xi_0)\not\in  \bigcap_{k>0} \WF _{\mathscr FL^\infty _{(\omega_k)}}  (f). 
 $$


\par

On the other hand, since the wave-front $\WF _{\mathscr FL^q_{(\omega )}} (f) $
is decreasing with respect to the parameter $q$, see Theorem \ref{wavefrontprop11},
we have
$$
 \bigcap_{k>0}  \WF _{\mathscr FL^\infty _{(\omega_{k })}}(f) \subseteq
 \bigcap_{k>0}  \WF _{\mathscr FL^q _{(\omega_{k})}}(f), \quad  q \in [1,\infty].
$$

\par
%
%
%

This completes the proof.
\end{proof}

\par

\section{Invariance properties of wave-front sets with respect to modulation spaces} \label{sec3}

\par

In this section we define wave-front sets with respect to modulation
spaces, and show that they coincide with wave-front sets of Fourier
Lebesgue types.

\par

\subsection{Modulation spaces} \label{subsec3}

\par

In this subsection we consider properties of
modulation spaces which will be used in microlocal analysis of ultradistributions.

\par

%
Let $ s>1$. For a fixed non-zero window   $\phi \in \cS ^{\{s\}} (\rr d )$
($\phi \in \cS ^{(s)} (\rr d )$ respectively) the \stft\ (STFT) of $f \in
 \cS ^{\{s\}} (\rr d ) $ ($f \in \cS ^{(s)} (\rr d )$ respectively)
 with respect to the window $\phi$ is given by
\begin{equation}
   \label{eqi2}
   V_\phi f(x,\xi)=
 (2\pi)^{-d/2} \int_{\rr d} f(y)\, {\overline {\phi(y-x)}} \, e^{-i\eabs{\xi,y}}\,dy\, ,
 \end{equation}

%
%

\begin{rem}
Throughout this section we consider only the case $ {\mathcal S}^{\{s\}} (\rr d ) $
and remark that the analogous assertions hold when
$ {\mathcal S}^{\{s\}} (\rr d ) $ is replaced by $ {\mathcal S}^{(s)} (\rr d )$.
\end{rem}

The map $(f,\phi )\mapsto V_\phi f$ from
$\mathcal{S} ^{\{s\}}(\rr d)\times \mathcal{S} ^{\{s\}}(\rr d)$ to
$\mathcal{S} ^{\{s\}}(\rr {2d})$
extends uniquely to a continuous mapping from
$(\mathcal{S} ^{\{s\}})'(\rr d)\times (\mathcal{S} ^{\{s\}})'(\rr d)$
to $(\mathcal{S} ^{\{s\}})'(\rr {2d})$ by duality.

\par

Moreover, for a fixed $\phi \in {\mathcal S}^{\{s\}} (\rr d )\setminus 0$,
$s\geq 1$, the following characterization of
$ {\mathcal S}^{\{s\}} (\rr d )$
holds:
\begin{equation}
f \in  {\mathcal S}^{\{s\}} (\rr d ) \quad  \Longleftrightarrow \quad
 V_\phi f  \in {\mathcal S}^{\{s\}} (\rr {2d} ).  \label{stft in s}
\end{equation}

\par

We refer to \cite{GZ, T2} for the proof and more details on STFT  in Gelfand-Shilov spaces.

\par
%
%
%

Now we recall the definition of modulation spaces. Let $s>1$, $\omega
\in \MRs (\rr {2d})$, $p,q\in
[1,\infty ]$, and the window $\phi \in \cS ^{(s)} (\rr d )\setminus 0$ be
fixed. Then the \emph{modulation space} $M^{p,q}_{(\omega )}(\rr d)$ is the set of all
ultra-distributions $f\in (\cS ^{\{s\}})' (\rr d) $  such that
\begin{equation}\label{modnorm}
\nm f{M^{p,q}_{(\omega )}} = \nm f{M^{p,q,\phi }_{(\omega )}} \equiv
\nm {V_\phi f\, \omega}{L^{p,q}_1}<\infty .
\end{equation}
Here $\nm \cdo {L^{p,q}_1}$ is the norm given by
\[
\nm F{L^{p,q}_1}\equiv \Big ( \int _{\rr {d}} \Big ( \int _{\rr
{d}}|F(x,\xi )|^p\, dx \Big )^{q/p} \, d\xi \Big
)^{1/q},
\]
when $F\in L^1_{loc}(\rr {2d})$ (with obvious interpretation when
$p=\infty$ or $q=\infty$). Furthermore, the modulation space
$W^{p,q}_{(\omega )}(\rr d)$ consists of all
 $f\in (\cS ^{\{s\}})' (\rr d) $
such that
\begin{equation*}
\nm f{W^{p,q}_{(\omega )}} = \nm f{W^{p,q,\phi }_{(\omega )}} \equiv
\nm {V_\phi f\, \omega}{L^{p,q}_2}<\infty ,
\end{equation*}
where $\nm \cdo {L^{p,q}_2}$ is the norm given by
\[
\nm F{L^{p,q}_2}\equiv \Big ( \int _{\rr {d}} \Big ( \int _{\rr
{d}}|F(x,\xi )|^q\, d\xi \Big )^{p/q} \, dx \Big )^{1/p},
\]
when $F\in L^1_{loc}(\rr {2d})$.

\par

If $\omega =1$, then the notations $M^{p,q}$ and $W^{p,q}$
are used instead of $M^{p,q}_{(\omega )}$ and $W^{p,q}_{(\omega )}$
respectively. Moreover we set $M^p_{(\omega )}=W^p_{(\omega )} =
M^{p,p}_{(\omega )}$ and $M^p=W^p = M^{p,p}$.

\par

We note that $M^{p,q}$ are modulation spaces of classical forms, while
$W^{p,q}$ are classical forms of Wiener amalgam spaces. We refer to
\cite{Feichtinger6} for the most updated definition of modulation
spaces.

\par


\par

If $ s>1 $, $p,q\in [1,\infty]$ and $\omega \in \MRs(\rr {2d})$, then one can show that
the spaces $\mathscr FL^q_{(\omega)}(\rr d)$,
$M^{p,q}_{(\omega )}(\rr d)$ and $W^{p,q}_{(\omega )}(\rr d)$ are locally the
same, in the sense that
\begin{multline*}
\mathscr FL^q_{(\omega )}(\rr d) \ttbigcap (\mathcal E^{\{s\}}) '(\rr d) =
M^{p,q}_{(\omega )}(\rr d) \ttbigcap (\mathcal E^{\{s\}})' (\rr d)\\[1 ex] =
W^{p,q}_{(\omega )}(\rr d) \ttbigcap (\mathcal E^{\{s\}})'(\rr d).
\end{multline*}
This follows by similar arguments as in  \cite{RSTT} (and replacing the space of polynomially
moderated weights  $ \mathscr P (\rr {2d})$  with $ \MRs(\rr {2d})$).
Later on we extend these properties in the context of
wave-front sets and recover the equalities above.

\par

The next proposition concerns topological questions of modulation spaces, and properties of the adjoint of the short-time Fourier transform.

Let $ s>1 $, $\omega \in \MRs(\rr {2d})$,
$ \phi \in  \mathcal{S} ^{(s)} \setminus 0$ and
$ F (x,\xi) \in L^{p,q} _{(\omega)} (\rr  {2d})$. Then $ V^* _\phi F $ is defined by the formula
\[
\langle V^* _\phi F, f \rangle  \equiv \langle F, V_\phi f \rangle, \qquad f\in \cS^{\{t\}} (\rr d).
\]
In what follows we let $L^{p,q}_{(\omega)}(\Omega)$, where $\Omega \subseteq \rr d$, be the set of all $F\in L_{loc}^1(\Omega)$ such that 
\[\nm{F}{L^{p,q}_{(\omega)}}\equiv\nm{F\cdot\omega \chi_{\Omega}}{L^{p,q}},\]
is finite, where $\chi_{\Omega}$ is the characteristic function of $\Omega$.

\par

\begin{prop} \label{emjedanve} {\rm \cite{CPRT1}}
Let $ s>1 $, $\omega \in \MRs(\rr {2d})$, $p,q\in [1,\infty]$, and
$ \phi , \phi _1\in  \mathcal{S} ^{(s)} (\rr d)$, with $(\phi , \phi _1)_{L^2}\neq 0$.
Then the following is true:
\begin{enumerate}
\item the operator $ V^* _\phi$ from $\mathcal{S} ^{(s)}(\rr {2d})$
to $\mathcal{S} ^{(s)}(\rr d)$ extends uniquely to a continuous operator from $L^{p,q} _{(\omega)} (\rr {2d})$ to $M^{p,q} _{(\omega)} (\rr {d})$, and
\begin{equation}  \label{vstar}
\|  V^* _\phi F  \|_{ M^{p,q} _{(\omega)} } \leq C \| V_{\phi _1} \phi \|_{ L^{1} _{(v)} }   \| F \|_{L^{p,q} _{(\omega)}} \text ;
\end{equation}
%
%

\vrum

\item $\Mopq (\rr d )$ is a  Banach space
whose definition is independent on the choice of window
$ \phi \in \mathcal{S} ^{(s)} \setminus  0 $;

\vrum

\item the set of windows can be extended from $ {\mathcal S} ^{(s)}(\rr d) \setminus  0$ to $M^1 _{(v)}(\rr d) \setminus  0$.
\end{enumerate}
\end{prop}

\par

%
%
%
%
%
%
%

%
%
%
%
%
%
%

\subsection{Wave-front sets with respect to modulation spaces}

\par

Next we define wave-front sets with respect to modulation spaces and show that they agree with corresponding wave-front sets of Fourier Lebesgue types. More precisely, we
prove that \cite[Theorem 6.1]{PTT} holds if the weights of polynomial growth
are replaced by more general submultiplicative weights.

\par

Let $s>1$, $\phi \in \mathcal{S} ^{\{s\}} (\rr d) \setminus 0$ ($\phi \in \mathcal{S} ^{(s)} (\rr d) \setminus 0$), $\omega \in
\MRs$, $\Gamma \subseteq \rr d\setminus 0$ be an open
cone and let $p,q\in [1,\infty ]$. For any $f\in ( \mathcal{S} ^{\{s\}} )'(\rr d)$ ($f\in ( \mathcal{S} ^{(s)} )'(\rr d)$)
we set
\begin{multline}\label{modseminorm}
|f|_{\mathcal B(\Gamma )} = |f|_{\mathcal B(\phi ,\Gamma )}
\equiv
\nm{V_\phi f}{L^{p,q}_{(\omega)}(\rr d \times \Gamma)}
\\[1ex]
\text{when}\quad \mathcal B=M^{p,q}_{(\omega )}=M^{p,q}_{(\omega
)}(\rr d).
\end{multline}
We note that
$ |f|_{\mathcal B(\phi ,\Gamma )}$ might attain $+\infty$.
Thus we define by $|\cdo |_{\mathcal B(\Gamma )}$ a semi-norm
on $( \mathcal{S} ^{\{s\}} )' (\rr d)$ (or on $( \mathcal{S} ^{(s)} )' (\rr d)$ ) which might attain the value $+\infty$. If $\Gamma
=\rr d\setminus 0$ and $\phi \in \mathcal{S} ^{\{s\}} (\rr d)$, then $|f|_{\mathcal B(\Gamma )} = \nm
f{M^{p,q}_{(\omega )}}$.

\par

We also set
\begin{multline}\label{modseminorm2}
|f|_{\mathcal B(\Gamma )} = |f|_{\mathcal B(\phi ,\Gamma )}
\equiv
\Big ( \int _{\rr {d}} \Big ( \int _{\Gamma} | V_\phi f(x,\xi )\omega
(x,\xi )|^q\, d\xi \Big )^{p/q}\, dx \Big )^{1/p}
\\[1ex]
\text{when}\quad \mathcal B=W^{p,q}_{(\omega )}=W^{p,q}_{(\omega
)}(\rr d)
\end{multline}
and note that similar properties hold for this semi-norm compared to
\eqref{modseminorm}.

\par

Let $\omega \in \MRs (\rr {2d})$, $p,q\in [1,\infty ]$, $f\in
( \mathcal{S} ^{(s)} )' (\rr d)$, and let
$\mathcal B=M^{p,q}_{(\omega )}$ or $\mathcal B=W^{p,q}_{(\omega
)}$. Then $\Theta _{\mathcal B}(f)$, $\Sigma
_{\mathcal B}(f)$ and the wave-front set $\WF
_{\mathcal B}(f)$ of $f$ with respect to the modulation space
$\mathcal B$ are defined in the same way as in Section
\ref{sec1}, after replacing the semi-norms of Fourier Lebesgue types in
\eqref{skoff1} with the semi-norms in \eqref{modseminorm} or
\eqref{modseminorm2} respectively.

\par

We need the following proposition when proving that the wave-front sets of Fourier-Lebesgue and modulation space types are the same. The result is an extension of  \cite[Proposition 4.2]{CPRT1}.

\par

\begin{prop} \label{stftestimates}
Let $ s>1$. Then the following is true:

\begin{enumerate}
\item if $ f \in (\mathcal E ^{\{s\}})' (\rr d)$ and $\phi \in \cS^{(s)}(\rr d)$, then 
\[  \label{rastiopadanje}
 |V_\phi f(x,\xi)|\lesssim e^{- h|x|^{1/s}} e^{\ep  | \xi |^{1/s}} ,
 \]
for every $h>0$ and every  $\ep >0$;
\item if $ f \in (\mathcal D ^{(s)})' (\rr d)$ and in addition $\phi \in \cD^{(s)}(\rr d)\setminus 0$, then $f\in (\mathcal E ^{\{s\}})' (\rr d)$, if and only if $\supp V_\phi f\subseteq K\times \rr d$ for some compact set $K$ and 
\[|V_\phi f(x,\xi)|\lesssim e^{\ep |\xi|^{1/s}},\]
for every $\ep >0$.
\end{enumerate}
\end{prop}

\begin{proof}
In order to prove $(1)$ we assume that $ f \in (\mathcal E ^{\{s\}})' (\rr d)$ and $\phi \in \cS^{(s)}(\rr d)$.
Also let $\psi\in \cD^{(s)}(\rr d)$ such that  $\psi=1$ in $\supp f$. Then for every $\ep, h >0$ it holds that 
\[
|V_\psi \phi(x,\xi)| \lesssim e^{-h |x|^{1/s}-2\ep |\xi |^{1/s}}.
\]
and
\[|\widehat{f}(\xi)|\lesssim e^{\ep |\xi|^{1/s}}.
\]
By straight-forward calculations, it follows that 
\begin{multline*}
|V_{\phi} f(x,\xi)| = |(V_{\phi} (\psi f))(x,\xi)| \lesssim (|V_{\psi}\phi(x,\cdot)|* |\widehat{f}|)(\xi)
\\[1ex]
= \int |V_{\psi}\phi (x,\xi - \eta )||
\widehat{f}(\eta )|\, d\eta \lesssim  \int e^{-h |x|^{1/s}-2\ep |\xi-\eta |^{1/s}} e^{\ep |\eta |^{1/s}}\, d\eta
\\[1ex]
\leq  e^{-h |x|^{1/s}} \int e^{-2\ep |\eta |^{1/s}+2\ep |\xi|^{1/s}+\ep |\eta |^{1/s}}\, d\eta
\end{multline*} 
Now, the assertion follows since both $\ep$ and $h$ can be chosen arbitrarily.

\par

Next we prove $(2)$. First assume that $\phi \in \cD^{(s)}(\rr d)\setminus 0$ and $f\in (\mathcal
E^{\{s\}})'(\rr d)$. Since both $\phi$ and $f$ have compact support, it follows that
$\supp (V_\phi f)\subseteq K\times \rr d$. Furthermore,
\[
|V_\phi f(x,\xi)|\lesssim e^{\ep(|x|^{1/s}+|\xi|^{1/s})}.
\]
Since $V_\phi f(x,\xi)$ has compact support in the $x$-variable, it follows that 
\[
|V_\phi f(x,\xi)|\lesssim e^{\ep|\xi|^{1/s}},
\]
for every $\ep >0$.

\par

In order to prove the opposite direction we assume that $\supp V_\phi f\subseteq K\times \rr d$,
for some compact set $K$, and 
\[|V_\phi f(x,\xi)|\lesssim e^{\ep |\xi|^{1/s}},\]
for every $\ep >0$. Then
\begin{equation}\label{fhatest2}
|\widehat{f}(\xi)|=\Big |\int V_\phi f(x,\xi)\, dx\Big|\lesssim e^{\ep|\xi|^{1/s}},\quad \forall \ep >0.
\end{equation}
Assume that $\supp \phi \subseteq K$ and choose $\fy \in \cD^{(s)}(\rr d)$ such that
$\supp \fy \cap 2K =\emptyset$. Let $x\notin 2K$, then 
\[
(f, \fy)=(V_\phi f, V_\phi \fy)=0,
\]
which implies that $f$ has compact support. Hence, \eqref{fhatest2} and the fact that
$f\in (\cD^{(s)})'(\rr d)$ give $f\in (\mathcal E^{\{s\}})'(\rr d)$.
\end{proof}

\par

\begin{thm}\label{wavefrontsequal}
Let $ s>1 $,  $p,q\in [1,\infty ]$,  $\omega \in \MRs(\rr {2d})$,
$\mathcal B=\mathscr FL^q_{(\omega)}(\rr d)$, and let $\mathcal C=M^{p,q}_{(\omega )}(\rr d)$ or
$\mathcal C=W^{p,q}_{(\omega )}(\rr d)$.
If $f\in (\mathcal{D}^{(s)})' (\rr d) $
then
\begin{equation}\label{WFidentities2}
\WF _{\mathcal B}(f)= \WF _{\mathcal C}(f).
\end{equation}
In particular, $\WF  _{\mathcal C}(f)$ is independent of $p$ and
$\phi \in \mathcal{S}^{(s)}(\rr d)\setminus 0$ in \eqref{modseminorm} and
\eqref{modseminorm2}.
\end{thm}

\par

In the proof of Theorem \ref{wavefrontsequal}, the main
part concerns proving that the wave-front sets of modulation types
are independent of the choice of window $\phi \in \cS^{(s)} (\rr d)\setminus 0$.
Note also that the dual pairing between  $f\in (\mathcal{S}^{\{s\}})' (\rr d) $ and
 $\phi \in \cS^{(s)} (\rr d) $ is well defined.

\par

\begin{proof} We only consider the case $\mathcal C = M^{p,q}_{(\omega )}$.
The case $\mathcal C = W^{p,q}_{(\omega )}$ follows by similar arguments and is
left for the reader. We may assume that $f\in (\mathcal E ^{\{s\}})' (\rr d)$
and that $\omega (x,\xi )=\omega (\xi )$ since the statements only concern local assertions.

\par

In order to prove that $\WF _{\mathcal C}(f)$ is independent of
$\phi \in \mathcal{S}^{(s)}(\rr d)\setminus 0$, we assume that $\phi ,\phi _1\in
 \mathcal{S}^{(s)}(\rr d)\setminus 0$ and let $|\cdo |_{\mathcal C_1(\Gamma )}$ be
the semi-norm in \eqref{modseminorm} after $\phi$ has been replaced by
$\phi _1$. Let $\Gamma _1$ and $\Gamma _2$ be open cones in $\rr d$
such that $\overline {\Gamma _2}\subseteq \Gamma _1$. The asserted
independency of $\phi$ follows if we prove that
\begin{equation}\label{est2.6}
|f|_{\mathcal C(\Gamma _2)} \le C(|f|_{\mathcal C_1(\Gamma _1)}+1),
\end{equation}
for some positive constant $C$.
%
%
%
Let
\[
\Omega _1=\sets {(x,\xi )}{\xi \in \Gamma _1}\subseteq \rr {2d}\quad
\text{and}\quad \Omega
_2=\complement \Omega _1\subseteq \rr {2d},
\]
with characteristic functions $\chi _1$ and $\chi _2$ respectively,
and set
$$
F_k(x,\xi )=|V_{\phi _1}f(x,\xi )|\omega (\xi )\chi _k(x,\xi ), \quad 
k=1,2,
$$ and $  G=|V_\phi \phi _1(x,\xi )|v(\xi )$.
Since $\omega$ is $v$-moderate, it follows from \cite[Lemma 11.3.3]{Gro-book} that
\[
|V_\phi f(x,\xi )\omega (x,\xi )|\lesssim \big ( (F_1+F_2)*G\big )(x,\xi),
\]
which implies that
\begin{equation}\label{fJ1J2igen}
|f|_{\mathcal C(\Gamma _2)} \lesssim J_1+J_2,
\end{equation}
where
\[
J_k = \Big (\int _{\Gamma _2} \Big (\int_{\rr d} |(F_k*G)(x,\xi )|^p\, dx\Big
)^{q/p}\, d\xi \Big )^{1/q}, \quad k=1,2.
\]

\par

By Young's inequality
\begin{equation}\label{J1estimateB}
J_1\le \nm {F_1*G}{L^{p,q}_1}\le \nm G{L^1}\nm {F_1}{L^{p,q}_1}
=C|f|_{\mathcal C_1(\Gamma _1)},
\end{equation}
where $C=\nm G{L^1} = \nm {V_\phi \phi _1(x,\xi )v(\xi )}{L^1}<\infty$, in view of \eqref{stft in s}.

\par

Next we consider $J_2$.
%
%
%
For $\xi\in \Gamma_2$ and $\eta\in \complement \Gamma_1$, it follows from \eqref{estimate1},
Propositon \ref{stftestimates} 
and \eqref{stft in s} that for every $N, l>0$ and for some $k>0$ we may choose
$h>0$ such that
\begin{multline*}
|(F_2 * G)(x,\xi )| \lesssim  \iint _{\rr {2d}}
e^{-N|y|^{1/s}} e^{(l+k)| \eta |^{1/s}}  e^{- h(|x-y|^{1/s} + |\xi - \eta|^{1/s})} v(\xi-\eta) \, dy d\eta
\\[1ex]\lesssim \iint _{\rr {2d}}
e^{-N|y|^{1/s}} e^{(l+k)| \eta |^{1/s}}
e^{- h c(|x|^{1/s} +  |y|^{1/s} + |\xi|^{1/s} + | \eta|^{1/s})/2}
e^{k(|\xi|^{1/s} + |\eta|^{1/s} )} \, dy d\eta
\\[1 ex]
\lesssim   e^{ -hc|x|^{1/s}/2}  e^{(k -\frac{hc}{2})|\xi|^{1/s}}
\iint _{\rr {2d}}  e^{-N|y|^{1/s} -hc |y|^{1/s}/2}
e^{(l +2k -hc/2)|\eta|^{1/s} } \, dy d\eta,
\\[1 ex]
\lesssim    e^{ -hc|x|^{1/s}/2}  e^{(k -hc/2)|\xi|^{1/s}} < \infty.
\end{multline*}
Therefore
\begin{multline*}
J_2  = \Big (\int _{\Gamma _2} \Big (\int_{\rr d} |(F_2 * G)(x,\xi )|^p\, dx\Big )^{q/p}\, d\xi \Big )^{1/q}
\\[1 ex]
\lesssim
 \Big (\int _{\Gamma _2} \Big (\int_{\rr d}
 e^{ -hc|x|^{1/s}/2}  e^{(k -hc/2)|\xi|^{1/s}}
 \Big)^p\, dx\Big )^{q/p}\, d\xi \Big )^{1/q}<\infty.
\end{multline*}
This proves that \eqref{est2.6}, and hence $\WF _{\mathcal C}(f)$ is
independent of $\phi \in \mathcal{S}^{(s)}(\rr d)\setminus 0$.

\medspace

In order to prove \eqref{WFidentities2} we assume from now on that
$\phi$ in \eqref{modseminorm} is real-valued and has compact support. Let
$p_0\in [1,\infty ]$ be such
that $p_0\le p$ and set $\mathcal C_0
=M^{p_0,q}_{(\omega )}$. The result follows if we prove
\begin{gather}
\Theta _{\mathcal C_0}(f)\subseteq \Theta _{\mathcal B}(f)\subseteq
\Theta _{\mathcal C}(f)\quad \text{when}\ p_0=1,\
p=\infty ,\label{Thetaest1}
\intertext{and}
\Theta _{\mathcal C}(f)\subseteq \Theta _{\mathcal
C_0}(f).\label{Thetaest2}
\end{gather}

\par
%
%

The proof of the first inclusion in \eqref{Thetaest1} follows from the estimates
\begin{multline*}
|f|_{\mathcal B(\Gamma )} \lesssim
\Big ( \int _{\Gamma} |\widehat f(\xi )\omega (\xi )|^{q}\, d\xi
\Big )^{1/q}
\\[1ex]
 \lesssim \Big ( \int _{\Gamma} |\mathscr F \Big ( f \int _{\rr {d}}\phi
(\cdot - x)\,  dx\Big ) (\xi ) \omega (\xi )|^{q}\, d\xi \Big
)^{1/q}
\\[1ex]
 \lesssim \Big ( \int _{\Gamma} \Big (  \int _{\rr {d}} |\mathscr F (
f \phi (\cdot -x)) (\xi ) \omega (\xi )|\,  dx \Big )^{q}\, d\xi \Big
)^{1/q}
\\[1ex]
=\Big ( \int _{\Gamma} \Big (  \int _{\rr {d}} |V_{\phi}
f(x,\xi )\omega (\xi )|\,  dx \Big )^{q} \, d\xi \Big )^{1/q} = C
|f|_{\mathcal C_0(\Gamma )},
\end{multline*}
for some positive constant $C$.
\par

Next we prove the second inclusion in \eqref{Thetaest1}. We have
\begin{multline*}
|f|_{\mathcal C(\Gamma _2)} = \Big ( \int _{\Gamma _2}
\sup_{x \in \rr d} | V_\phi f(x,\xi )\omega (x,\xi ) |^{q}\, d\xi \Big
)^{1/q}
\\[1ex]
\lesssim \Big ( \int _{\Gamma _2} \sup_{x \in  \rr d}
|( |\widehat f| * | \mathscr F (\phi (\cdot - x)) | ) (\xi) \omega
(\xi )|^{q}\, d\xi \Big )^{1/q}
\\[1ex]
\lesssim \Big ( \int _{\Gamma _2}
|( |\widehat f| * | \widehat \phi  | ) (\xi) \omega (\xi )|^{q}\,
d\xi \Big )^{1/q}
\\[1ex]
\lesssim  \Big ( \int _{\Gamma _2} \big ( ( |\widehat f \cdot  \omega
| * |\widehat \phi \cdot v| ) (\xi) \big )^{q} \, d\xi \Big )^{1/q},
\end{multline*}
where $\phi\in\mathcal D^{(s)}(X)$ is chosen such that $\phi=1$ in $\supp f$. The second inclusion in \eqref{Thetaest1} now follows by straight-forward computations, using similar arguments as in the proof of \eqref{cuttoff1}. The details are left for the reader.


\par

It remains to prove \eqref{Thetaest2}. Let $K\subseteq \rr d$ be compact and chosen such that $V_\phi f(x,\xi)=0$ outside K, and let $p_1\in [1,\infty]$ be chosen such that $1/p_1 + 1/p_0=1+1/p$. By
H{\"o}lder's inequality we get
\begin{multline*}
|f|_{\mathcal C_0(\Gamma )} = \Big ( \int _{\Gamma} \Big
( \int _{\rr {d}} | V_\phi f(x,\xi )\omega (x,\xi )|^{p_0}\, dx\Big
)^{q/p_0}\, d\xi \Big )^{1/q}
\\[1ex]
\le C_K\Big ( \int _{\Gamma} \Big (
\int _{\rr {d}} | V_\phi f(x,\xi )\omega (x,\xi )|^{p}\, dx\Big
)^{q/p}\, d\xi \Big )^{1/q} = C_K|f|_{\mathcal C(\Gamma )}.
\end{multline*}
This gives \eqref{Thetaest2}, and the proof is complete.

\end{proof}

\par

\begin{cor}
Let $s>1$, $p,q\in [1,\infty ]$, and $\omega \in \MRs(\rr {2d})$.
If  
$f\in  (\mathcal{S}^{\{s\}})' (\rr d)$ is compactly supported,
then
\[
f \in \mathcal B \quad
\Longleftrightarrow \quad
\WF _{\mathcal B }(f) =\emptyset,
\]
where $\mathcal B$ is equal to $\mathscr FL^{q}_{(\omega)}$,
$M^{p,q}_{(\omega )}$ or $W^{p,q}_{(\omega )}$.
\end{cor}

\par

In particular, we recover Corollary 6.2 in \cite{PTT}, Theorem 2.1 and Remark 4.6 in \cite{RSTT}.

\par


\section{Discrete semi-norms in Fourier Lebesgue spaces} \label{sec4}

\par

In this section we introduce discrete analogues of the semi-norms in \eqref{skoff1} and \eqref{modseminorm}, and show that these semi-norms are finite if and only if the corresponding non-discrete semi-norms are finite. The techniques used here are similar to those in \cite{JPTT1}. 

\par

Assume that $q\in [1,\infty]$, $s>1$, $\omega\in \MRs(\rr d)$, $\mathcal B =\mathscr F\!L^q_{(\omega)}(\rr d)$, and $H\subset \rr d$ is a discrete set. Then we set
\begin{equation*}
|f|^{(D)}_{\mathcal B(H)} \equiv \left( \sum_{\xi_l\in H} |\widehat{f} (\xi_l) \omega(\xi_l)|^q \right)^{1/q},
\qquad \widehat{f}\in C(\rr d)\cap (\mathcal S ^{\{s\}})'(\rr d)
\end{equation*}
with obvious modifications when $q=\infty$. As in the continuous case, we may allow weight functions in
$ \MRs(\rr {2d})$, i.{\,}e. $\omega=\omega(x,\xi)$.
However, again we note that the condition
\begin{equation*}
|f|^{(D)}_{\mathcal B(H)} < \infty
\end{equation*}
is independent of $x\in \rr d$.

\par

By a lattice $\Lambda$ we mean the set
\begin{equation*}
\Lambda =\{ a_1 e_1 +\cdots +a_d e_d; \, a_1,\dots , a_d \in \mathbf Z\},
\end{equation*}
where $e_1,\dots,e_d$ is a basis in $\rr d$.

\par
The following Lemma was proved for distributions, cf. \cite{JPTT1,RT1,RT2}.

\begin{lemma}\label{discretelemma1}
Let $s>1$, $f\in (\mathcal E ^{\{s\}})'(\rr d)$, $\Gamma$ and $\Gamma_0$ be open cones in $\rr d\setminus 0$ such that $\overline{\Gamma_0}\subseteq \Gamma$, $q\in [1,\infty]$, and let $\Lambda\subseteq \rr d$ be a lattice. If $|f|_{\mathcal B(\Gamma)}$ is finite, then $|f|^{(D)}_{\mathcal B (\Gamma_0 \cap \Lambda)}$ is finite.
\end{lemma}

\begin{proof}
We only prove the result for $q<\infty$, leaving the small modifications in the case $q=\infty$ for the reader. Assume that $|f|_{\mathcal B(\Gamma)}<\infty$, and let $H=\Gamma_0\cap \Lambda$. Also let $\fy\in \mathcal D^{(s)}(\rr d)$ be such that $\fy =1$ in $\supp f$. Then
\begin{multline*}
(|f|^{(D)}_{\mathcal B (\Gamma_0 \cap \Lambda)})^q
= \sum_{\xi_l \in H} |\mathscr F (\fy f)(\xi_l) \omega (\xi_l)|^q
\\[1 ex]
= (2\pi)^{-qd/2} \sum_{\xi_l \in H} \left | \int \widehat{\fy}(\xi_l-\eta)\widehat{f}(\eta) \omega(\xi_l) \, d\eta\right |^q
\lesssim (S_1+S_2),
\end{multline*}
where
\begin{align*}
S_1 = & \sum_{\xi_l \in H} \left( \int_{\Gamma}             \psi(\xi_l-\eta)F(\eta) \, d\eta \right)^q,
\\[1 ex]
S_2 = & \sum_{\xi_l \in H}\left( \int_{\complement\Gamma}\psi(\xi_l-\eta)F(\eta) \, d\eta \right)^q .
\end{align*}
Here we set $F(\xi)=|\widehat{f}(\xi)\omega(\xi)|$ and $\psi(\xi)=|\widehat{\fy}(\xi)v(\xi)|$ as in the proof of Theorem \ref{wavefrontprop11}.
\par

We need to estimate $S_1$ and $S_2$. Let  $c$ be chosen such that $\omega$ is  moderate  with respect to 
$v = e^{c |\cdot|^{1/s}}.$  By Hölder's inequality we get
\begin{multline*}
S_1\leq C' \sum_{\xi_l \in H} \left ( \int_{\Gamma} \psi(\xi_l-\eta) F(\eta)\, d\eta\right)^q
\\[1 ex]
= C' \sum_{\xi_l \in H} \left ( \int_{\Gamma} \psi(\xi_l-\eta)^{1/q'}( \psi(\xi_l-\eta)^{1/q}F(\eta))\, d\eta\right)^q
\\[1 ex]
\leq  C' \nm{F}{L^1}^{q/q'}\sum_{\xi_l \in H} \int_{\Gamma} \psi(\xi_l-\eta)F(\eta))^q \, d\eta
\\[1 ex]
\leq C''\int_{\Gamma} F(\eta)^q \, d\eta
= C''|f|^q_{\mathcal B(\Gamma)},
\end{multline*}
where
\begin{equation*}
C'' = C' \nm{\psi}{L^1}^{q/q'} \sup_{\eta \in \rr d} \sum_{\xi_l \in H}  \psi(\xi_l-\eta)
\end{equation*}
is finite by \eqref{psiexpest}. This proves that $S_1$ is finite.

\par

It remains to prove that $S_2$ is finite. We observe that
\begin{multline*}
|\xi_l-\eta|^{1/s} \geq c \max (|\xi_l|^{1/s}, |\eta|^{1/s}) \geq \frac{c}{2} (|\xi_l|^{1/s} + |\eta|^{1/s})
\\[1 ex]
\text{when} \quad \xi_l \in H \quad\text{and} \quad \eta \in \complement \Gamma.
\end{multline*}
Since $f\in (\mathcal E^{\{s\}})'(\rr d)$ it follows that
\begin{equation*}
|F|\lesssim e^{c|\cdot|^{1/s}}
\end{equation*}
for every positive constant $c$. Furthermore, since $\fy \in \mathcal D^{(s)}(\rr d)$, it follows that for each $N\geq 0$, there is a constant $C_N$ such that
$\psi \leq C_N e^{-N|\cdot|^{1/s}}.$ This gives
\begin{multline*}
S_2 \lesssim  \sum_{\xi_l \in H} \left( \int_{\complement \Gamma} e^{-N|\xi_l-\eta|^{1/s}} e^{k|\eta|^{1/s}} \, d\eta \right)^q
\\[1 ex]
\lesssim  \sum_{\xi_l \in H}
e^{-q N c/2|\xi_l|^{1/s}}
\left( \int e^{-(N c/2-k)|\eta|^{1/s}} \, d\eta \right)^q,
\end{multline*}
where we have used the fact that $\omega $ is $v$-moderate.
The result now follows, since the right-hand side is finite when $N > 2k/c.$
The proof is complete.
\end{proof}

\par

Next we prove a converse of Lemma \ref{discretelemma1},
in the case when the lattice $\Lambda$ is dense enough. Let $e_1,\dots,e_d$ in $\rr d$ be a basis for $\Lambda$, i.{\,}e. for some $x_0\in \Lambda$ we have
\begin{equation*}
\Lambda=\{x_0 + t_1 e_1 + \cdots + t_d e_d; \, t_1,\dots, t_d\in \mathbf Z\}.
\end{equation*}
A parallelepiped $D$, spanned by $e_1,\dots,e_d$ for $\Lambda$ and with corners in $\Lambda$, is called a $\Lambda$-parallelepiped. This means that for some $x_0\in \Lambda$ and some basis $e_1,\dots,e_d$ for $\Lambda$ we have
\begin{equation*}
D=\{x_0 + t_1 e_1 + \cdots + t_d e_d; \, t_1,\dots, t_d\in [0,1]\}.
\end{equation*}

\par

We let $\mathcal A(\Lambda)$ be the set of all $\Lambda$-parallelepipeds. For future references we note that if $D_1, D_2\in \mathcal A(\Lambda)$, then their volumes $|D_1|$ and $|D_2|$ agree, and for convenience we let $\| \Lambda\|$ denote the common value, i.{\,}e.
\begin{equation*}
\|\Lambda\|=|D_1|=|D_2|.
\end{equation*}

\par

Let $\Lambda_1$ and $\Lambda_2$ be lattices in $\rr d$ with bases $e_1,\dots, e_d$ and $\ep_1,\dots,\ep_d$ respectively. Then the pair $(\Lambda_1,\Lambda_2)$ is called \emph{admissible lattice pair}, if for some $0<c\leq 2 \pi$ we have $\eabs{ e_j,\ep_j}=c$ and $\eabs{ e_j,\ep_k}=0$ when $j\neq k$. If in addition $c< 2\pi$, then $(\Lambda_1,\Lambda_2)$ is called a \emph{strongly admissible lattice pair}. If instead $c=2 \pi$, then the pair $(\Lambda_1,\Lambda_2)$ is called a \emph{weakly admissible lattice pair}.

\par

\begin{lemma}\label{discretelemma2}
Let $s>1$, $(\Lambda_1,\Lambda_2)$ be an admissible lattice pair, $D_1 \in \mathcal A(\Lambda_1)$ (be open), and let
$f \in (\mathcal E ^{\{s\}})'(\rr d )$ be such that an open neighbourhood of its support is contained in $D_1$.
Also let $\Gamma$ and $\Gamma_0$ be open cones in $\rr d$ such that $\overline{\Gamma_0}\subseteq \Gamma$.
If $|f|^{(D)}_{\mathcal B(\Gamma\cap \Lambda_2)}$ is finite, then $|f|_{\mathcal B(\Gamma_0)}$ is finite.
\end{lemma}

\begin{proof}
Since $D_1$ contains an open neighbourhood of the support of $f$, we may modify $\Lambda_1$ (and therefore $D_1$)
such that the lattice pair $(\Lambda_1,\Lambda_2)$ is strongly admissible,
and such that the hypothesis still holds. From now on we therefore assume that $(\Lambda_1,\Lambda_2)$
is strongly admissible.

We use similar arguments as in the proof of Lemma \ref{discretelemma1}.
Again we prove the result only for $q < \infty$. The small modifications to the case $q = \infty$ are left for
the reader.

Assume that $|f|^{(D)}_{\mathcal B(\Gamma \cap\Lambda_2)} < \infty$, and let $\fy\in  \mathcal D^{(s)}
(D_1^\circ)$ be equal to one in the support of $f$, where $D_1^\circ$ denotes the interior of
the set $D_1$.

By expanding $f = \fy f$ into a Fourier series on $D_1$ we get
\begin{equation*}
\widehat{f}(\xi)= C \sum_{\xi_l\in \Lambda_2} \widehat{\fy} (\xi - \xi_l) \widehat{f}(\xi_l),
\end{equation*}
where the positive constant $C$ only depends on $\Lambda_2$. We have
\begin{multline*}
(|f|_{\mathcal B(\Gamma_0)})^q = \int_{\Gamma_0} |\widehat{f}(\xi)\omega(\xi)|^q \, d\xi
\\[1 ex]
= C^q \int_{\Gamma_0}\left| \sum_{\xi_l \in \Lambda_2} \widehat{\fy}(\xi-\xi_
l)\widehat{f}(\xi_l)\omega(\xi) \right|^q \, d\xi \leq C^q (S_1+S_2),
\end{multline*}
where
\begin{align*}
S_1 = &\int_{\Gamma_0}\left| \sum_{\xi_l \in H_1} \widehat{\fy}(\xi-\xi_l)\widehat{f}(\xi_l)\omega(\xi) \right|^q \, d\xi,
\\[1 ex]
S_2= &\int_{\Gamma_0}\left| \sum_{\xi_l \in H_2} \widehat{\fy}(\xi-\xi_l)\widehat{f}(\xi_l)\omega(\xi) \right|^q \, d\xi,
\end{align*}
$H_1 = \Gamma \cap \Lambda_2 $ and $H_2 = \complement \Gamma \cap \Lambda_2. $

\par

We have to estimate $S_1$ and $S_2$.  Let $\omega$ be moderate  with respect to the weight
$v(\cdot) = e^{k |\cdot|^{1/s}}.$  By Minkowski's inequality we get
\begin{multline*}
S_1 \leq  C \int_{\Gamma_0}\left(\sum_{\xi_l \in H_1} | \widehat{\fy}(\xi - \xi_l)v(\xi - \xi_l)||\widehat{f} (\xi_l)\omega(\xi_l )|\right)^q \, d\xi
\\[1 ex]
\leq  C' \int_{\Gamma_0}\left(\sum_{\xi_l \in H_1} | \widehat{\fy}(\xi - \xi_l)v(\xi - \xi_l)||\widehat{f} (\xi_l)\omega(\xi_l )|^q\right) \, d\xi
\\[1 ex]
\leq C'' \sum_{\xi_l \in H_1} |\widehat{f} (\xi_l)\omega(\xi_l )|^q ,
\end{multline*}
where
\begin{equation*}
C' = C \sup_\xi \nm{\widehat{\fy}(\xi - \xi_l)v(\xi - \xi_l)}{l^1(\Lambda_2)}^{q/q'}
< \infty,
\quad \text{and} \quad  C'' = C' \nm{\fy}{\mathscr F \! L^1_{(v)}}< \infty.
\end{equation*}
This proves that $S_1$ is finite when $|f|^{(D)}_{\mathcal B(\Gamma\cap \Lambda_2)} < \infty.$

\par

It remains to prove that $S_2$ is finite. We recall that
\begin{multline*}
|\xi - \xi_l|^{1/s} \geq c \max(|\xi |^{1/s}, |\xi_l |^{1/s}) \geq c(|\xi |^{1/s}+ |\xi_l |^{1/s})/2
\\[1 ex]
\text{when} \quad \xi\in \Gamma_0 \ \text{and} \ \xi_l \in H_2,
\end{multline*}
and use the same arguments as in the proof of Lemma  \ref{discretelemma1}
to obtain
\begin{multline*}
S_2 \lesssim  \int_{\Gamma_0} \left( \sum_{\xi_l\in H_2}e^{-N |\xi-\xi_l|^{1/s}} e^{k |\xi_l|^{1/s}} \right)^q \, d\xi
\\[1 ex]
\lesssim  \int_{\Gamma_0} e^{- q N c/2 |\xi|^{1/s}} \left( \sum_{\xi_l\in H_2}
e^{-(N c/2-k) |\xi_l|^{1/s}} \right)^q \, d\xi.
\end{multline*}
The result now follows, since the right-hand side is finite when $N>2k/c$. The proof is complete.
\end{proof}

\par

\begin{cor}\label{cuttoffdiscrete} Let $s>1$, $(\Lambda_1,\Lambda_2)$ be an admissible lattice pair,
$D_1 \in \mathcal A(\Lambda_1)$, and let $f \in (\mathcal E^{\{s\}})'(\rr d)$ be such that an open
neighbourhood of its support is contained in $D_1$. Also let $\Gamma$ and $\Gamma_0$ be open cones in $\rr d$
such that $\overline{\Gamma_0}\subseteq \Gamma$. If $|f|^{(D)}_{\mathcal B(\Gamma\cap \Lambda_2)}$ is finite,
then $|\fy f |^{(D)}_{\mathcal B(\Gamma_0\cap \Lambda_2)}$ is finite for every $\fy \in \mathcal D^{(s)}(X)$.
\end{cor}

For the proof we recall that $|\fy f |_{\mathcal B(\Gamma_0)}$ is finite when
$f \in (\mathcal E^{\{s\}})' (\rr d )$, $\fy\in \mathcal D^{\{s\}}(X)$,
and $|f|_{\mathcal B(\Gamma)}$ is finite. This follows from the proof of Theorem \ref{wavefrontprop11}.

\begin{proof} Let $\Gamma_1$, $\Gamma_2$ be open cones such that $\overline{\Gamma_j}\subseteq \Gamma_{j+1}$ for $j = 0, 1$, $\overline{\Gamma_2}\subseteq\Gamma$, and assume that $|f|^{(D)}_{\mathcal B(\Gamma\cap\Lambda_2)} < \infty$. Then Lemma \ref{discretelemma2} shows that $|f|_{\mathcal B(\Gamma_2) }$ is finite. Hence, Theorem \ref{wavefrontprop11} implies that $|\fy f |_{\mathcal B(\Gamma_1)} < \infty$. This gives $|\fy  f|^{(D)}_{\mathcal B(\Gamma_0 \cap \Lambda_2)} < \infty$, in view of Lemma  \ref{discretelemma1}. The proof is complete.
\end{proof}

\par


\section{Gabor pairs}

\par

In this section we introduce in Definition \ref{Gaborpair} the notion of Gabor pairs.
We refer to \cite{JPTT1} for an explanation that conditions in Definition \ref{Gaborpair} are quite general.

\par

By Definition \ref{Gaborpair} it follows that
our analysis can be applied to the most general classes
of non-quasianalytic ultradistributions, and it also points out the role of Beurling-Domar weights
in definitions of $ \mathscr FL^q_{(\omega )}(\rr d) $ and $ M^{p,q}_{(\omega )}(\rr d)$,
cf. \cite{Fe4, Gro2, GZ}.
On the other hand, a larger class of quasianalytic
ultradistributions can not be treated by the technique given here, since the corresponding test function spaces
do not contain smooth functions of compact support.

\par

Assume that $e_1,\dots, e_d$ is a basis for the lattice $\Lambda_1$, and that $(\Lambda_1,\Lambda_2)$ is a weakly admissible lattice pair. If $f \in L^2_{\loc}$ is periodic with respect to $\Lambda_1$, and $D$ is the parallelepiped, spanned by $\{e_1, \dots , e_d \}$, then we may make Fourier expansion of $f$ as
\begin{equation}\label{f1}
f (x) = \sum_{\xi_l\in \Gamma_2} c_l e^{i \eabs{x,\xi_l}}, \qquad x \in \rr d
\end{equation}
(with convergence in $L^2_{\loc}$), where the coefficients $c_l$ are given by
\begin{equation}\label{coeff1}
c_l =\int_{\Delta} f (y)e^{-i \eabs{y,\xi_l}}\, dy.
\end{equation}
Here and in what follows we let
\begin{equation}\label{scalpro}
y = y_1e_1 + \cdots + y_de_d, \qquad dy = dy_1 \cdots dy_d \qquad \text{and} \qquad
\Delta = [0, 1]^d .
\end{equation}
For non-periodic functions and distributions we instead make Gabor expansions. Because of the support properties of the involved Gabor atoms and their duals, we are usually forced to change the assumption on the involved lattice pairs. More precisely, instead of assuming that $(\Lambda_1,\Lambda_2)$ should be a weakly admissible lattice pair, we assume from now on that $(\Lambda_1,\Lambda_2)$ is a strongly admissible lattice pair, with $\Lambda_1 = \{x_j\}_{j\in J}$ and $\Lambda_2 = \{\xi_l \}_{l\in J}$. Also let $s>1$ and
\begin{align}
\nonumber\phi,\psi \in \mathcal D^{(s)}(\rr d ), \qquad\phi_{j,l} (x) = &\phi(x - x_j )e^{i \eabs{x,\xi_l}}
\\[1 ex]
\text{and} \quad \psi_{j,l} (x) = \psi(x - x_j )e^{i \eabs{x,\xi_l}}\label{phipsi1}
\end{align}
be such that $\{\phi_{j,l} \}_{j,l\in J}$ and $\{\psi_{j,l}\}_{j,l\in J}$ are dual Gabor frames (see \cite{Gro-book,fegr97} for the definition and basic properties of Gabor frames and their duals). If $f \in(\mathcal S^{\{s\}})'(\rr d )$ then
\begin{equation}\label{f}
f = \sum_{j,l\in J} c_{j,l}\phi_{j,l},
\end{equation}
where
\begin{equation}\label{coeff}
c_{j,l} = C_{\phi,\psi}( f,\psi_{j,l})_{L^2(\rr d )}
\end{equation}
and the constant $C_{\phi,\psi}$ depends on the frames only.

\par

Note that the convergence is in  $(\mathcal S^{\{s\}})'(\rr d )$ due to Proposition \ref{stftestimates}.


\par

\begin{defn}\label{Gaborpair}
Assume that $\ep \in(0, 1]$, $\{x_j \}_{j \in J} = \Lambda_1\subseteq \rr d$ and
$\{\xi_l \}_{l\in J} = \Lambda_2 \subseteq  \rr d$ are lattices and let $\Lambda_1(\ep) = \ep \Lambda_1$.
Also let  $\phi,\psi\in C_0 ^\infty (\rr d )$ be non-negative, and set
\begin{align}\label{phipsi}
&\phi^{\ep} = \phi( \cdot /\ep), &\psi^{\ep}&=\psi( \cdot /\ep),
\\[1 ex] \nonumber
&\phi^{\ep}_{j,l} = \phi^{\ep}(\cdot -\ep x_j )e^{i\eabs{\cdot,\xi_l}}, &\psi^{\ep}_{j,l} &=\psi^{\ep}(\cdot -\ep x_j )e^{i\eabs{\cdot,\xi_l}}
\end{align}
when $\ep x_j \in \Lambda_1(\ep)$ (i.{\,}e. $x_j \in \Lambda_1$) and $\xi_l \in \Lambda_2$. Then the pair
\begin{equation}\label{pair}
(\{\phi_{j,l}\}_{j,l\in J} , \{\psi_{j,l }\}_{j,l\in J} )
\end{equation}
is called a Gabor pair with respect to the lattices $\Lambda_1$ and $\Lambda_2$ if for each $\ep \in (0, 1]$, the sets $\{\phi^{\ep}_{j,l }\}_{j,l \in J}$ and $\{\psi^{\ep}_{j,l}\}_{j,l\in J}$ are dual Gabor frames.
\end{defn}

\par

By Definition \ref{Gaborpair} and Chapters 5-13 in \cite{Gro-book} it follows that if $f \in (\mathcal S^{\{s\}})'(\rr d )$ and if
$(\{\phi_{j,l}\}_{j,l \in J} , \{\psi_{j,l} \}_{j,l \in J} )$ is a Gabor pair, then
\begin{equation}\tag*{(\ref{f})$''$}
f = \sum_{j,l \in J} c_{j,l}(\ep) \phi^{\ep}_{j,l}
\end{equation}
in $(\mathcal S^{\{s\}})'(\rr d )$, for every $\ep \in (0, 1]$, where
\begin{equation}\tag*{(\ref{coeff})$''$}
c_{j,l}( \ep) = (f,\psi^{\ep}_{j,l}).
\end{equation}

\par

Here $(\cdot, \cdot)$ denotes the unique extension of the $L^2$-form on $\mathcal S^{\{s\}}(\rr d ) \times \mathcal S^{\{s\}}(\rr d )$ into
$(\mathcal S^{\{s\}})'(\rr d )\times (\mathcal S^{\{s\}})'(\rr d )$.

\par

We remark that if the pair in \eqref{pair} is a Gabor pair, then it follows
from the investigations in \cite{Gro-book} that the lattice pair $(\Lambda_1,\Lambda_2)$ in Definition \ref{Gaborpair} is strongly admissible.

\par

The following proposition explains that any pair of dual Gabor frames
satisfying a mild additional condition, generates a Gabor pair.

\par

\begin{prop}\label{propgabor} {\rm{\cite{JPTT1}}}
Let $\phi ,\psi \in C_0 ^\infty (\rr d )$ be non-negative functions
and let $\phi_{j,l}$ and $\psi_{j,l}$
be given by \eqref{phipsi1}.
Also, let $\Lambda_1$ and $\Lambda_2$ be the same as in Definition \ref{Gaborpair}.
If  $\{ \phi_{j,l} \} _{j,l\in J}$ and $\{\psi _{j,l} \} _{j,l\in J}$ are dual Gabor frames such that
\begin{equation} \label{normL1}
\sum_{x_j \in \Lambda_1} \phi(\cdot - x_j)\psi(\cdot - x_j)=\|\Lambda_1\|^{-1},
\end{equation}
holds, then \eqref{pair} is a Gabor pair.
\end{prop}

\par

\begin{rem}
If $\phi = \psi$, then \eqref{normL1} describes the tight frame property of the corresponding Gabor frame, cf. \cite[Theorem 6.4.1]{Gro-book}.
\end{rem}

\par

\begin{rem}\label{diskrmodnorm}
Let $p,q\in [1,\infty ]$, $\omega \in \MRs(\rr {2d})$, and $f \in (\mathcal E^{\{s\}})'(\rr d )$. If 
$(\{ \phi _{j,l} \} _{j,l\in J}, \{ \psi_{j,l} \} _{j,l\in J} )$ is a Gabor pair such that \eqref{f} and \eqref{coeff} hold, then it follows that $f\in M^{p,q}_{(\omega )}(\rr d)$  if and only if
$$
\nm f{[\ep ]}\equiv \Big (\sum _{l\in J} \Big ( \sum _{j\in J}|c_{j,l}(\ep
)\omega (\ep x_j,\xi _j)|^p \Big )^{q/p} \Big )^{1/q}
$$
is finite for every $\ep \in (0,1]$. Furthermore, for every $\ep \in (0,1]$, the norm $f\mapsto
\nm f{[\ep ]}$  is equivalent to the modulation space norm (1.3) (cf. \cite{Fe4,FG1,FG2,Gro-book}.)
\end{rem}

\par

\section{Discrete versions of wave-front sets}\label{sec6}

\par

In this section we define discrete wave-front sets with respect to Fourier Lebesgue and
modulation spaces, and prove that they agree with the corresponding wave-front sets
of continuous types. In the first part we consider discrete versions with respect to Fourier
Lebesgue and modulation spaces, and show that they agree with each other, and also
with the corresponding continuous ones. In the second part we consider more general
situations, where we discuss similar questions for sequences of spaces. In such a way
we are able to characterize Hörmander's wave-front sets with our discrete approach.

\par

\subsection{Discrete versions of wave-front sets with
respect to Fourier Lebesgue and modulation spaces} 

\par

We start with two definitions.

\par

\begin{defn}\label{discFourLebWF}
Let $s>1$, $q \in [1,\infty]$, $f \in (\mathcal S^{\{s\}})'(\rr d)$,
and let $(\Lambda_1,\Lambda_2)$ be a strongly admissible lattice pair in $\rr d$ such that
$x_0  \not\in \Lambda_1$. Moreover, let $\omega \in \MRs(\rr d)$ and
$\mathcal B=\mathscr FL^q_{(\omega )}(\rr d)$. Then the discrete wave-front set $\DF  _{\mathcal B}(f)$
consists of all $(x_0,\xi _0)\in X \times (\rr d\setminus 0),$
$X \subseteq \rr d$ is open such that for each
$\fy \in\mathcal D^{(s)}(X)$ with $\fy (x_0)\neq 0$ and each open conical neighbourhood $\Gamma $ of $\xi
_0$, it holds
$$
|\fy f|_{\mathcal B(\Gamma )}^{(D) } = \infty .
$$
\end{defn}

\par

For the definition of discrete wave-front sets of modulation spaces, we consider
Gabor pairs $(\{ \phi _{j,l} \}
_{j,l\in J}, \{ \psi _{j,l} \} _{j,l\in J} )$,
and let
$$
J_{x_0}(\ep )=J_{x_0}(\ep ,\phi ,\psi )=J_{x_0}(\ep ,\phi ,\psi
,\Lambda _1)
$$
be the set of all $j\in J$ such that
$$
x_0\in \supp \phi _{j,l}^\ep \quad \text{or}\quad x_0\in \supp \psi
_{j,l}^\ep.
$$

\par

\begin{defn}\label{discModWF}
Let $s>1$, $p,q\in [1,\infty ]$, $f\in (\mathcal S^{\{s\}})'(X)$,
and let
 $\phi,\psi\in \mathcal D^{(s)} (\rr d )$ be non-negative such that
 $(\{ \phi _{j,l} \} _{j,l\in J}, \{ \psi_{j,l} \} _{j,l\in J} )$ is a Gabor pair with respect to
the lattices $\Lambda _1$ and $\Lambda _2$ in $\rr d$. Moreover, let $\omega \in \MRs(\rr {2d})$
and $\mathcal C=M^{p,q}_{(\omega )}(\rr d)$. Then the discrete wave-front set $\DF  _{\mathcal C}(f)$
consists of all $(x_0,\xi _0)\in X \times (\rr d\setminus 0),$
$X\subseteq \rr d$ is open,
such that for each $\ep \in (0,1]$  and each
open conical neighbourhood $\Gamma $ of $\xi _0$, it holds
$$
\Big (\sum _{ \xi _l \in \Gamma \cap \Lambda _2}\Big (\sum
_{j\in J_{x_0}(\ep )}|c_{j,l}(\ep )\omega (\xi _l)|^p\Big )^{q/p}\Big
)^{1/q} = \infty ,
$$
where
$$
f=\sum _{j,l\in J}c_{j,l (\ep)} \phi _{j,l}
^\ep,\quad  c_{j,l} (\ep) = C_{\phi ,\psi}
(f,\psi_{j,l}^\ep )_{L^2(\rr d)}
$$
and the constant $C_{\phi ,\psi}$
depends on  $\phi$ and $\psi$ only.
\end{defn}

\par

Roughly speaking, $(x_0,\xi _0)\in \DF  _{\mathcal C}(f)$ means that $f$ is not locally in $\mathcal C$, in the direction $\xi _0$.
The following result shows that our wave-front sets coincide.

\par

\begin{thm} \label{dWFsets}
Let $s>1$, $X\subseteq \rr d$ be open and let $f\in (\mathcal D^{\{s\}})'(X)$.
Then
\begin{equation}\label{WFidentity}
\WF_{\mathcal B}(f)=\WF_{\mathcal C}(f)=\DF_{\mathcal B}(f)=\DF_{\mathcal C}(f).
\end{equation}
\end{thm}

\par

\begin{proof}
By Theorem \ref{wavefrontsequal} and Lemmas \ref{discretelemma1} and
\ref{discretelemma2}, it follows that the first two equalities in
\eqref{WFidentity} hold. The result therefore follows if we prove
that $\DF _{\mathcal B}(f) = \DF _{\mathcal C}(f)$.

\par

First assume that $(x_0,\xi _0)\notin \DF _{\mathcal B}(f)$, and choose $\fy \in \mathcal D^{(s)}(X)$, where $1<t<s$, an open neighbourhood
$X_0\subset \overline{X_0}\subset X$ of $x_0$ and conical neighbourhoods $\Gamma ,\Gamma
_0$ of $\xi _0$ such that
\begin{itemize}
\item
$
\overline{\Gamma _0}\subseteq \Gamma , \quad \fy (x)=1 \quad
\text{when}\quad x\in X_0 ,$
\item
$|\fy \, f|_{\mathcal B(H)}^{(D) }<\infty  ,\quad \text{when}\quad H=\Lambda _2\cap \Gamma .$
\end{itemize}
Now let $(\{ \phi _{j,l} \} _{j,l\in J}, \{ \psi _{j,l} \}
_{j,l\in J} )$ be a Gabor pair and choose $\ep \in (0,1]$
such that $\supp \phi _{j,l}^\ep$ and $\supp \psi _{j,l}^\ep$ are
contained in $X_0$ when $x_0\in \supp \phi _{j,l}^\ep$ and $x_0\in
\supp \psi _{j,l}^\ep$. Since
$$
c_{j,l}(\ep ) =C(f,\psi _{j,l}^\ep )_{L^2(\rr d)} = \mathscr F(f\,
\psi (\cdo /\ep -x_j))(\xi _l),
$$
it follows from these support properties that if $H_0=\Lambda _2\cap
\Gamma _0$, then
\begin{multline}\label{fchisums}
\Big ( \sum _{\xi _l \in H_0} |\mathscr F(f\, \psi (\cdo /\ep
-x_j))(\xi _l)\omega (\xi _l)|^q\Big )^{1/q}
\\[1ex]
= |f\, \psi (\cdo /\ep -x_j)|_{\mathcal B(H_0)}^{(D) }
= |f\, \fy \psi (\cdo /\ep -x_j)|_{\mathcal B(H_0)}^{(D) },
\end{multline}
when $j\in J_{x_0}(\ep )$. Hence, by combining Corollary
\ref{cuttoffdiscrete} with the facts that $J_{x_0}(\ep )$ is finite
and $|\fy \, f|_{\mathcal B(H)}^{(D) }<\infty$, it
follows that the expressions in \eqref{fchisums} are finite and
$$
\Big ( \sum _{ \xi _l \in H_0} \Big (\sum _{j\in J_{x_0}(\ep
)}|\mathscr F(f\, \psi (\cdo /\ep -x_j))(\xi _l)\omega (\xi _l)|^p\Big
)^{q/p}\Big )^{1/q}<\infty .
$$
This implies that $(x_0,\xi _0)\notin \DF _{\mathcal C}(f)$,
and we have proved that $\DF _{\mathcal C}(f) \subseteq
\DF _{\mathcal B}(f)$.

\par

In order to prove the opposite inclusion we assume that $(x_0,\xi
_0)\notin \DF _{\mathcal C}(f)$, and we choose $\ep \in (0,1]$,
Gabor pair $(\{ \phi _{j,l} \} _{j,l\in J}, \{ \psi
_{j,l} \} _{j,l\in J} )$ and conical neighbourhoods $\Gamma ,\Gamma
_0$ of $\xi _0$ such that $\overline{\Gamma _0}\subseteq \Gamma$ and
\begin{equation}\label{discrmodfinitecond}
\Big ( \sum _{ \xi _l \in H} \Big (\sum _{j\in J_{x_0}(\ep
)}|\mathscr F(f\, \psi (\cdo /\ep -x_j))(\xi _l)\omega (\xi _l)|^p\Big
)^{q/p}\Big )^{1/q}<\infty ,
\end{equation}
when $H=\Lambda _2\cap \Gamma$. Also choose $\fy ,\kappa \in \mathcal D^{(s)}(X)$
such that $\fy (x_0)\neq 0$ and
$$
\kappa (x)\sum _{j\in J_{x_0}(\ep )} \psi (x/\ep-x_j)=1,
\quad
\text{when}\quad x\in \supp \fy .
$$
Since $J_{x_0}(\ep )$ is finite, H{\"o}lder's inequality gives
\begin{multline*}
|\fy \, f|_{\mathcal B(H_0)}^{(D) } =\Big |\sum
_{j\in J_{x_0}(\ep )} (\fy \kappa ) \, (f \, \psi (\cdo /\ep -x_j))\Big
| _{\mathcal B(H_0)}^{(D) }
\\[1ex]
\le \Big ( \sum _{ \xi _k \in H_0} \Big (\sum _{j\in J_{x_0}(\ep
)}|\mathscr F((\fy \kappa )f\, \psi (\cdo /\ep -x_j))(\xi _l)\omega
(\xi _l)|\Big )^{q}\Big )^{1/q}
\\[1ex]
\lesssim\Big ( \sum _{ \xi _l \in H_0} \Big (\sum _{j\in J_{x_0}(\ep
)}|\mathscr F((\fy \kappa ) f\, \psi (\cdo /\ep -x_j))(\xi _l)\omega
(\xi _l)|^p\Big )^{q/p}\Big )^{1/q},
\end{multline*}
where $H_0 = \Lambda_2 \cap \Gamma_0$.
By Corollary \ref{cuttoffdiscrete} and \eqref{discrmodfinitecond} it
now follows that the right-hand side in the last estimates is
finite. Hence, $|\fy \, f|_{\mathcal B(H_0)}^{(D) }<\infty $, which shows that $(x_0,\xi _0)\notin
\DF _{\mathcal B}(f)$, and we have proved that
$\DF _{\mathcal B}(f)\subseteq \DF _{\mathcal C}(f)$. The proof is complete.
\end{proof}

\par

In the following corollary we give a discrete description of the $s$-wave-front set,
$\WF_{s} (f),$ from Section \ref{sec2}.

\begin{cor} Let $q \in [1,\infty],$ $ s>1$, and let $ \omega_k (\xi) \equiv e^{k|\xi|^{1/s}} $ for 
$ \xi \in \rr d$ and  $k>0$. If $f \in (\mathcal{E}^{\{s\}})' (\rr d)$, then
\begin{equation}
\bigcap_{k>0} \DF _{\mathscr FL^q _{(\omega_k )}} (f) =  \WF_{s} (f).
\label{intersection2}
\end{equation}
\end{cor}

We remark that a discrete analogue of $\WF_{s} (f)$ also can be introduced in a similar way as
in \cite{RT1, RT2}. Let us denote this set by $\WF_{s,T} (f) $, and refer to it as
\emph{toroidal} $s$-wave-front set.
It can be proved that
\begin{equation}  \label{toroid}
\WF_{s,T} (f) = \mathbf{T}^d \times  \mathbf{Z}^d  \cap \WF_s (f),
\end{equation}
where $\mathbf{T}^d $ is the torus in $\mathbf{R}^d$.

\par

A significant difference between the toroidal wave-front sets and our discrete wave-front sets
lies in the fact that $\WF _T(f)$ only informs about the \emph{rational} directions
for the propagation of singularities of $f$ at a certain point,
while $\DF  (f)=\WF (f)$ takes care of \emph{all} directional for $f$ to that point, we refer to \cite{JPTT1} for an example.

\par

\end{document}